\documentclass[11pt]{amsart}
\usepackage{amssymb,latexsym, amsmath, enumerate, amsthm, mathtools}
\usepackage{color}
\usepackage{graphicx,overpic}
\usepackage{hyperref}
\usepackage{comment}
\hypersetup{
      CJKbookmarks=true
	colorlinks=true,
	linkcolor=blue,
	filecolor=blue,      
	urlcolor=blue,
	citecolor=blue,
}

\usepackage[margin=1.5in]{geometry}

\usepackage{setspace}
\newcommand\R{\mathbb R}
\newcommand\N{\mathbb N}

\newcommand\G{\mathbb G}
\newcommand\g{\mathfrak{g}}
\newcommand\X{\mathrm X}
\newcommand\Y{\mathrm Y}
\newcommand\D{\mathcal D}

\newcommand\MCP{\mathrm {MCP}}
\newcommand\GEO{\mathrm {GEO}}
\newcommand\CE{\mathrm {CE}}

\newcommand{\Cut}{\mathcal{C}}

\newcommand\ol{\overline}
\newcommand\w{\widetilde}

\newcommand\rf{\mathrm{f}}
\newcommand\rN{\mathrm{N}}
\newcommand\scr{\mathcal}

\newtheorem{theorem}{Theorem}
\newtheorem{lemma}[theorem]{Lemma}
\newtheorem{remark}[theorem]{Remark}
\newtheorem{corollary}[theorem]{Corollary}
\newtheorem{proposition}[theorem]{Proposition}
\newtheorem{definition}[theorem]{Definition}
   
\newtheorem{conjecture}[theorem]{Conjecture}

\begin{document}

\title[Lower bound of curvature exponent]{On the lower bound of the curvature exponent on step-two Carnot groups}

\author[Y. Zhang]{Ye Zhang}
\address[Y.~Zhang]{Analysis on Metric Spaces Unit, Okinawa Institute of Science and Technology Graduate University, Okinawa 904-0495, Japan; Scuola Internazionale Superiore di Studi Avanzati (SISSA), Via Bonomea 265, 34136 Trieste, Italy} 
\email{yezhang@sissa.it}
\thanks{This project has received funding from (i)~the European Research Council (ERC) under the European Union's Horizon 2020 research and innovation programme (grant agreement GEOSUB, No. 945655); (ii)~the JSPS Grant-in-Aid for Early-Career Scientists (No. 24K16928).}


\begin{abstract}
In this work, we show that there exists a step-two Carnot group on which the new lower bound of the curvature exponent given in \cite{NZ24} can be strictly less than the curvature exponent by studying the convergence of the structure constants of Lie algebra.
\end{abstract}

\maketitle

\section{Introduction}

\medskip

\subsection{Measure contraction property}

First we recall that measure contraction property is some weak variant of the curvature-dimension condition, which is introduced in \cite{S06} or \cite{O07}. In the following we assume the underlying  space is a length space $(X,d,\mu)$ with negligible cut loci, i.e. a metric measure space with distance function $d$,  Borel measure $\mu$, and  for every $x \in X$, there exists a negligible set $\Cut (x)$ and a measurable function $\Phi^x: X \setminus \Cut (x) \times [0,1] \to X$, such that the curve $t \mapsto \Phi^x(y,t)$ is the unique length minimizing geodesic joining $x$ and $y$. On such space, we define the set of $s$-intermediate points by
\[
Z_s(x,A) = \{\Phi^x(y,s):  y \in A \setminus \Cut (x)\}, \quad \forall \, s \in [0,1], x \in X, A \ \mbox{measurable}.
\]

\begin{definition}[Measure contraction property]
A length space $(X,d,\mu)$ with negligible cut loci satisfies $\MCP(K,N)$ for some $K \in \R$ and $N > 1$, or $K \le 0$ and $N = 1$, if for every $x \in X$ and measurable $A \subset X$
	with $0<\mu(A)<\infty$ (also $A \subset B(x,\pi \sqrt{(N - 1)/K})$ if $K > 0$), we have
\[
	\mu(Z_s(x,A)) \ge \int_A s \left[ \frac{s_K(s d(x,y)/\sqrt{N - 1})}{s_K(d(x,y)/\sqrt{N - 1})} \right]^{N - 1} d\mu(y).
	\qquad\forall \, s\in[0,1],
	\]
where we adopt the convention that $0/0 = 1$ and the term in the square bracket is $1$ for $N = 1$. Here and in the following $B(x,r)$ is the open ball centered at $x$ with radius $r$, and the function $s_K$ is defined by
\[
s_K(t) := \begin{cases}
(1/\sqrt{K}) \sin(\sqrt{K} t) \quad & \mbox{if $K > 0$}, \\
t \quad & \mbox{if $K = 0$}, \\
(1/\sqrt{-K}) \sinh(\sqrt{-K} t) \quad & \mbox{if $K < 0$}.
\end{cases}
\]
\end{definition}

Here we collect some properties of the measure contraction property. For more details, we refer to \cite{O07}.
\begin{proposition}[Lemma 2.4 and Theorem 4.3 of \cite{O07}]\label{pMCP}
On general length space $(X,d,\mu)$, the following properties hold: 
\begin{enumerate}[(i)]
   \item \label{monoto}
    $\MCP(K,N)$ implies $\MCP(K',N')$ for $K' \le K$ and $N' \ge N$; 
   \item  \label{chasc}
   If $\MCP(K,N)$ holds on $(X,d,\mu)$, then $\MCP(a^{-2}K,N)$ holds on $(X,a \, d, b \, \mu)$ for $a,b > 0$;
   \item \label{GBM}
   (Generalized Bonnet--Myers theorem) If $\MCP(K,N)$ holds with $K > 0$ and $N > 1$, then the underlying space is bounded. 
\end{enumerate}
\end{proposition}

\subsection{Step-two Carnot groups}

Recall that a connected and simply connected Lie group $\G = (\G, \cdot)$ is a step-two Carnot group
if its left-invariant Lie algebra $\mathfrak{g}$ admits a stratification
\begin{align*}
\mathfrak{g} = \mathfrak{g}_1 \oplus \mathfrak{g}_2, \quad
[\mathfrak{g}_1, \mathfrak{g}_1] = \mathfrak{g}_2, \quad
[\mathfrak{g}_1, \mathfrak{g}_2] = \{0\},
\end{align*}
where $[\cdot,\cdot]$ denotes the Lie bracket on $\mathfrak{g}$. Via the Lie group exponential map, we can regard $(\G, \cdot)$ as $(\g, \ast)$ with the group operation $\ast$ on $\g$ given by
\begin{align}\label{mulalg}
a \ast b := a + b + \frac{1}{2} [a,b],  \qquad \forall \, a,b \in \g.
\end{align}
Note that on $(\g,\ast)$ the identity is $0$ and the inverse element of $a$ is $- a$.
For more details, we refer to \cite{BLU07}.

\begin{definition}[Sub-Riemannian structure]
\label{Sub-Riemannian structure}
On a step-two Carnot group $\G$, given an inner product $\langle \cdot, \cdot \rangle$ of $\g_1$, the canonical left-invariant sub-Riemannian structure $(\D,g)$ is defined as follows: the horizontal distribution $\D$ (a sub-bundle of the tangent bundle $T\G$ satisfying the bracket generating condition) is generated by $\g_1$ and the metric $g$ on $\D$ is determined by the inner product $\langle \cdot, \cdot \rangle$. To be more precise, let $m := \dim \g_1$ and assume $\{\X_1,\ldots, \X_m\}$ is an orthonormal basis of $\g_1$ under $\langle \cdot, \cdot \rangle$. Then we have
\[
\D_p := \mathrm{span} \{\w{\X}_1(p),\ldots, \w{\X}_m(p)\},
\]
where
\begin{align}\label{leftin}
\w{\X}_j (p) = D L_p (e) \X_j, \qquad \forall \, p \in \G,
\end{align}
with $L_p(q) := p \cdot q$, $e$ is the identity element on $\G$, and $D$ denotes the differential.
Furthermore, $\{\w{\X}_1(p),\ldots, \w{\X}_m(p)\}$ forms an orthonormal basis of $\D_p$ at every point $p \in \G$. 
\end{definition}

\begin{remark}
In this work we regard the Lie algebra $\g$ as the tangent space of $\G$ at the identity element $e$, namely $T_e \G$. So by \eqref{leftin}, $\w{\X}_j$ are just the left-invariant vector fields on $\G$.
\end{remark}

An absolutely continuous path $\gamma : [0,1] \to \G$ is called horizontal if $\dot{\gamma}(s) \in \D_{\gamma(s)}$ for a.e. $s \in [0,1]$, whose length can be calculated by 
\begin{equation}
\label{lenght}
\ell(\gamma) := \int_0^1 \sqrt{g(\dot{\gamma}(\tau),\dot{\gamma}(\tau))} \, d\tau.
\end{equation}
\begin{definition}[Carnot--Carath\'eodory distance]
\label{CC distance}
The Carnot--Carath\'eodory distance between two points $p, q \in \G$ is defined as
\[
d(p,q) := \inf\{\ell(\gamma): \gamma \mbox{ horizontal}, \gamma(0) = p, \gamma(1) = q\}.
\]
\end{definition}
By the celebrated Rashevskii--Chow Theorem (see for example \cite{ABB20, M02}), $d$ is a well-defined finite distance and its induced topology is the same as the manifold topology. In particular, this means that $d$ is continuous with respect to the standard Euclidean topology. Therefore $(\G,d)$ is a metric space. A length minimizing geodesic is an horizontal path $\gamma$ such that $\ell(\gamma) = d(\gamma(0), \gamma(1))$.

By definition it is easy to see that the Carnot--Carath\'eodory distance is left-invariant:
\begin{align}\label{leftd}
d(p_* \cdot p, p_* \cdot q) = d(p,q), \qquad \forall \, p_*, p ,q \in \G.
\end{align}

For every $\lambda > 0$, the dilation $\delta_\lambda$ on the Lie algebra $\g$ is the linear map such that $\delta_\lambda (a) = \lambda^\ell a$ for $a \in \g_\ell$, $\ell = 1,2$. It is actually a Lie algebra automorphism so it induces the dilation on $\G$, which is a Lie group automorphism and we still denote it by $\delta_\lambda$ as well. We have the following relation for the Carnot--Carath\'eodory distance $d$ and the dilation:
\begin{align}\label{hgd}
d(\delta_\lambda(p), \delta_\lambda(q)) = \lambda d(p,q), \qquad \forall \, p, q \in \G, \lambda > 0.
\end{align}

Finally the Haar measure  $\mu$ on $\G$ turns out to be the pushforward measure of the Lebesgue measure $\nu$ on $\g$ via the Lie group exponential map. With a fixed choice of the Haar measure $\mu$, $(\G, d, \mu)$ is a metric measure space. Furthermore, we have 
\[
\mu(\delta_\lambda(A)) = \lambda^Q \mu(A), \qquad \forall \, A \ \mbox{measurable}, \lambda > 0,
\]
with homogeneous dimension $Q: = \dim \g_1 + 2 \dim \g_2 = m +  2 \dim \g_2$, which is strictly larger than the topological dimension  $n := \dim \g_1 + \dim \g_2 = m +  \dim \g_2$.

\subsection{Measure contraction property on step-two Carnot groups}

By \cite[Proposition 15]{R13} we know that step-two Carnot groups (equipped with the Haar measure) are length spaces with negligible cut loci. From \eqref{GBM} of Proposition \ref{pMCP}, we cannot expect $\MCP(K,N)$ holds on $\G$ with positive $K$ while  \eqref{chasc} of Proposition \ref{pMCP} tells us that the $\MCP(K,N)$ is independent of the choice of the Haar measure $\mu$. Furthermore, from the dilation structure, $(\G,d,\mu)$ is isomorphic with $(\G,\lambda^{-1} d,\lambda^{-Q} \mu)$. Then \eqref{chasc} of Proposition \ref{pMCP} also implies that if $\MCP(K,N)$ holds with some negative $K$, then $\MCP(\lambda^2 \, K,N)$ holds for every $\lambda > 0$, which in turn implies $\MCP(0,N)$ holds after taking the limit. However, $\MCP(0,N)$ is stronger than $\MCP(K,N)$ for $K < 0$ by \eqref{monoto} of Proposition \ref{pMCP}, so we only focus on the study of $\MCP(0,N)$. 

\begin{definition}[Curvature exponent]\label{defce}
We define $N_{\CE}$, the curvature exponent of $\G$, as the minimum of all $N$ such that $\MCP(0,N)$ holds on $\G$. Equivalently, $N_{\CE}$ is the minimum of all $N$ such that for every $p\in \G$ and measurable $A \subset \G$ with $0< \mu(A)<\infty$, we have
	\[
	\mu(Z_s(p,A)) \ge s^N \mu(A), 
	\qquad\forall \, s \in[0,1].
	\]
\end{definition}

From the argument above, we know that $\MCP(0,N_{\CE})$ is the best measure contraction property on step-two Carnot group $(\G,d,\mu)$ in the sense that if $\MCP(K,N)$ holds, then we must have $K \le 0$ and $N \ge N_{\CE}$.

The first result about the measure contraction property in the setting of step-two Carnot groups  is $N_{\CE} = 2Q - n = 3n - 2m$ on Heisenberg groups by Juillet in \cite{J09}. Then similar results are established on corank $1$ Carnot groups and generalized H-type groups in \cite{R16} and \cite{BR18} respectively.

For general step-two Carnot groups, the finiteness of the curvature exponent $N_{\CE}$, is obtained on ideal Carnot groups in \cite{R13} and generalized for all step-two Carnot groups:

\begin{theorem}[Theorem 1.4 of \cite{BR20}]\label{tmcp}
Given a step-two Carnot group $\G$, there exists a constant $N > 1$ such that the measure contraction property $\mathrm{MCP}(0,N)$ holds.
\end{theorem}

Note that it is shown recently in \cite{J21, RS23, MR23} that sub-Riemannian manifolds (equipped with smooth, positive measure) cannot admit any curvature-dimension condition. From Theorem \ref{tmcp}, the measure contraction property is a reasonable substitute of such condition in the setting of step-two Carnot groups. Regarding the curvature exponent $N_{\CE}$,  in \cite{R16}, the author proposed the following conjecture.

\begin{conjecture}\label{con1}
In the setting of Carnot groups with negligible cut loci, we have $N_{\CE} = N_{\GEO}$.
\end{conjecture}

Here the geodesic dimension $N_{\GEO}$ is defined by
\[
\sup_{p \in \G}\left[ \inf\left\{ N > 0 :
\sup_{A \in \scr F} \limsup_{s \to 0} \frac{\mu(Z_s(p,A))}{s^N \mu(A)} = \infty 
\right\} \right],
\]
where $\scr F := \{A \subset \G \text{ bounded, measurable with } 0< \mu(A) <\infty\}$. Note that from left-invariance, the outer supremum does not play any role and we can choose any $p$ in the group $\G$, in particular the identity element $e$. For an equivalent definition of the geodesic dimension (at a point), we refer to \cite[Definition 5.47]{ABR18}. Similarly, in our setting the point will not play any role. The equivalence is a result of \cite[Theorem D]{ABR18} there.

However, Conjecture \ref{con1} is not true, which is shown in \cite{NZ24}. More precisely, we found a new lower bound of the curvature exponent, which we denote by $N_0$, and a step-two Carnot group on which $N_{\CE} \ge N_0 > N_{\GEO}$. As a result, there is a gap between $N_{\CE}$ and $N_{\GEO}$. Naturally, we have the following new conjecture.

\begin{conjecture}\label{con2}
In the setting of Carnot groups with negligible cut loci, we have $N_{\CE} = N_0$.
\end{conjecture}

\subsection{Main result}

The main target of this work is to show that Conjecture \ref{con2} is not true as well. However, it is hard to construct a counterexample directly like the one in \cite{NZ24}. Instead, we illustrate this by showing that the new lower bound $N_0$ is not lower semicontinuous w.r.t. the convergence of structure constants of Lie algebra (see Definition \ref{lac} and Remark \ref{sc} below) while the curvature exponent $N_\CE$ is. In the following set $\N = \{1, 2, \ldots\}$ and we use $\ol{\N}$ to denote the set $\N \cup \{\infty\}$. 

\begin{definition}\label{lac}
We say a sequence of Lie algebra $\{(\g,[\cdot,\cdot]_k)\}_{k \in \N}$ converges to the limit Lie algebra $(\g,[\cdot,\cdot]_\infty)$ if the followings hold:

1. $\{(\g,[\cdot,\cdot]_k)\}_{k \in \ol{\N}}$ has the same Lie algebra stratification, i.e.
\begin{align*}
\mathfrak{g} = \mathfrak{g}_1 \oplus \mathfrak{g}_2, \quad
[\mathfrak{g}_1, \mathfrak{g}_1]_k = \mathfrak{g}_2, \quad
[\mathfrak{g}_1, \mathfrak{g}_2]_k = \{0\}, \quad \forall \, k \in \ol{\N};
\end{align*}

2. for any $a,b \in \g$, we have $[a,b]_k \to [a,b]_\infty$ as $k \to \infty$.

\end{definition}

\begin{remark}\label{sc}
Fix a basis $\{\X_1 , \ldots, \X_{m}\}$ of $\g_1$ and a basis  $\{\Y_1 , \ldots, \Y_{n - m}\}$ of $\g_2$ respectively. For any $1 \le i,j \le m$, we have
\[
[\X_i, \X_j]_k = \sum_{\ell = 1}^{n - m} c_{ij,k}^\ell \Y_\ell, \qquad \forall \, k \in \ol{\N}.
\]
From this equation we know that the convergence of the Lie algebra in Definition \ref{lac} is equivalent to  the convergence of the structure constants 
\[
c_{ij,k}^\ell \to c_{ij,\infty}^\ell, \qquad \forall \, 1 \le i,j \le m, 1 \le \ell \le n - m
\]
as $k \to \infty$.
\end{remark}

For any fixed $k \in \ol{\N}$, the Lie algebra $[\cdot,\cdot]_k$ induces a group structure $\ast_k$ on $\g$ by 
\[
a \ast_k b := a + b + \frac{1}{2} [a,b]_k,  \qquad \forall \, a,b \in \g, k \in \ol{\N}.
\]
In the following we denote $(\g, \ast_k)$ simply by $\G_k$. Furthermore, given some fixed inner product $\langle \cdot, \cdot \rangle$ on $\g_1$ (independent of $k \in \ol{\N}$), we can define a canonical left-invariant sub-Riemannian structure as before and we use $d_k$ to denote the Carnot--Carath\'eodory distance of such sub-Riemannian structure on $\G_k$. 

For $k \in \ol{\N}$, we define $\rf_k \in \Gamma(\g_1^* \otimes T \g)$ by letting $\rf_k|_p$ to be a linear map from $\g_1$ to $T_p \g$ given by $\rf_k|_p(v) := D L_{p, \G_k}(0) v$ for each $p \in \g$. Here $L_{p, \G_k}(q) := p \ast_k q$, $k \in \ol{\N}$. Since $\g$ is a vector space, we can identify $T_p \g$ with $\g$ itself (like we do for vector spaces, not for Lie groups since we have several group structures here) and after the identification (which is also the bundle trivialization) we have
\[
\rf_k|_p(v) =  v + \frac{1}{2} [p,v]_k, \qquad \forall \, v \in \g_1, p \in \g, k \in \ol{\N}.
\]
From the expression above it is easy to know that $\{\rf_k\}_{k \in \ol{\N}} \subset \mathrm{Lip}\Gamma(\g_1^* \otimes T \g)$ and $\rf_k \to \rf_\infty$ in the topology of $\mathrm{Lip}\Gamma(\g_1^* \otimes T \g)$. Here $\mathrm{Lip}\Gamma(\g_1^* \otimes T \g)$ denotes the locally Lipschitz sections of $\g_1^* \otimes T \g$. For the precise definition and the topology on such space, we refer to \cite[\S~2.1]{ALN23}. Now define a continuous varying norm $\rN$ on $\g \times \g_1$ by
\[
\rN(p,v) := \sqrt{\langle v, v \rangle}.
\]
Then it follows from \cite[Corollary 3.19]{ALN23} that $d_k = d_{(\rf_k,\rN)}$ for $k \in \ol{\N}$ and thus the main result there can be applied to our setting. The result is the following.

\begin{proposition}[(iv) of Theorem 1.4 of \cite{ALN23}]\label{t1}
Assume $\{(\g,[\cdot,\cdot]_k)\}_{k \in \N}$ converges to $(\g,[\cdot,\cdot]_\infty)$ as $k \to \infty$. Then we have $d_k(p,q) \to d_\infty(p,q)$ uniformly on every compact set of $\g \times \g$.  
\end{proposition}

\begin{remark}
In the first version of the manuscript Proposition \ref{t1} is proved by comparing the endpoint maps of different group structures. For the upper semicontinuity of $k \mapsto d_k(p,q)$ w.r.t. the Lie algebra convergence, we used the uniform heat kernel bounds under generalized curvature-dimension inequalities in \cite{BBG14, BG17}, which is not really needed.  As pointed out by the referee, such upper semicontinuity is obtained in \cite{ALN23} with the help of degree theory or quantitative open-mapping type result under suitable smoothness assumption. See also \cite[\S~3.3.4]{ABB20} for such kind of argument.
\end{remark}

One result of Proposition \ref{t1} is the pointed measured Gromov--Hausdorff convergence (see \cite[Definition 3.24]{GMS15}), which in turn implies the stability of the measure contraction property $\MCP(K,N)$ under the Lie algebra convergence. 

\begin{corollary}[\cite{GMS15} and Lemma 2.5 of \cite{BMRT24}]\label{cmcp}
Assume $\{(\g,[\cdot,\cdot]_k)\}_{k \in \N}$ converges to $(\g,[\cdot,\cdot]_\infty)$ as $k \to \infty$ and fix a  Lebesgue measure $\nu$ on $\g$ (independent of $k \in \ol{\N}$). Then $(\G_k, d_k, \nu, 0)$ converges to $(\G_\infty, d_\infty, \nu, 0)$ in the pointed measured Gromov--Hausdorff sense  as $k \to \infty$. 

As a result, if assume further that for every $k \in \N$, $\MCP(K,N)$ holds on $(\G_k, d_k, \nu)$, then $\MCP(K,N)$ also holds on $(\G_\infty, d_\infty, \nu)$. In particular, $N_\CE$ is lower semicontinuous w.r.t. such Lie algebra convergence.
\end{corollary}

Now we can disprove Conjecture \ref{con2} by  showing $N_0$ is not lower semicontinuous. In fact, from the construction of our counterexample we can show more. Here is our main result.

\begin{theorem}\label{c1}
There exists an ideal Carnot group such that $N_{\CE} > N_0$.
\end{theorem}

Recall that a Carnot group is called ideal if it does not admit non-trivial abnormal minimizing geodesics (see also Definition \ref{defM} and Remark \ref{idealM} below). We also refer to \cite{R13, R16} for more details about the definition as well as properties of ideal Carnot groups.

\begin{remark}
Recall that in \cite{NZ24} the lower bound $N_0$ actually comes from the small time asymptotic behavior of the Jacobian of the sub-Riemannian exponential map. In fact it is also related with the small time asymptotic behavior of sub-Laplacian of the geodesic cost function associated with strongly normal geodesic. For more details we refer to \cite[Theorem C]{ABR18}, together with \cite[Theorem B and Remark 4.11]{ABR18}, while a related equivalent characterization of $\MCP(0,N)$ is given in \cite[Proposition 2.5]{BR20}. Thus this corollary says local properties of the Jacobian (or geodesic cost function) cannot characterize the curvature exponent, which is global in nature. 
\end{remark}

\begin{remark}
The phenomenon of the gap between the curvature exponent and the geodesic dimension is also known for the measure contraction property on sub-Finsler Heisenberg groups recently. See the works \cite{BT23, BMRT242, BMRT24}. Also in this setting, the small time behavior cannot characterize the curvature exponent. See \cite[(ii) of Remark 5.13]{BMRT242}, together with \cite[Example 5.14]{BMRT242} or \cite[Theorem 1.2]{BMRT24}. 
\end{remark}

We guess $N_\CE$ is actually continuous w.r.t. such Lie algebra convergence and as a result it could be a non integer number while the lower bound $N_0$ is always an integer. However, it is hard to prove since we do not know the cut time of every geodesic on step-two Carnot groups. Until now, the cut times are only known on multi-dimensional Heisenberg groups \cite[\S~13.2]{ABB20}, corank $2$ groups \cite{BBG12}, free step-two Carnot group with $3$ generators \cite{MM17}, and Reiter--Heisenberg groups \cite{MM23}. We also refer the interested reader to \cite{RS17} for the conjecture about the cut locus on free step-two Carnot groups and \cite{MM25} for the recent pregresses on free step-two Carnot group with $4$ generators.

\begin{conjecture}\label{con3}
There exists a step-two Carnot group on which $N_{\CE}$ is not an integer.
\end{conjecture}

\subsection{Structure of the paper} In Section \ref{sec2} we recall the lower bound $N_0$ introduced in \cite{NZ24} and the proof of Theorem \ref{c1} is provided in Section \ref{sec4}.

\section{The lower bound $N_0$}\label{sec2}

\medskip

Fix a step-two Carnot group $\G$. To give the explicit formula for $N_0$ given in \cite{NZ24} in the setting of step-two Carnot groups, we need to first extend the given inner product $\langle \cdot,\cdot \rangle$ on $\g_1$ to the whole Lie algebra $\g$ in such way that $\g_1$ and $\g_2$ are orthogonal. Although different extension will result in different linear maps and spaces (to be defined below), it turns out that the number $N_0$ is the same. Since we only focus on the number $N_0$ in this work so we can just fix one such inner product on $\g$. For the result of this subsection, we refer to \cite[\S~6.4]{NZ24}. 

For every $u \in \g_2$, we can define a linear map $J_u : \g_1 \to \g_1$ by the following formula:
\begin{align}\label{defJ}
\langle u, [v,w] \rangle = \langle J_u v, w \rangle, \qquad \forall \, v,w \in \g_1.
\end{align}

From definition $u \mapsto J_u$ is linear, continuous, and $J_u$ is skew-symmetric since
\[
\langle J_u v, w \rangle = \langle u, [v,w] \rangle = -\langle u, [w,v] \rangle = - \langle J_u w, v \rangle =  - \langle  v, J_u w\rangle, \quad \forall \, v,w \in \g_1.
\]
Furthermore, since $[\g_1,\g_1] = \g_2$, for $u \in \g_2 \setminus \{0\}$ we must have $J_u \ne 0$.

Now for fixed $\g \ni p = \xi + u$ with $\xi \in \g_1$ and $u \in \g_2$, we define
\begin{align}\label{spg1}
U^\ell(p) := \mathrm{span} \{\xi, J_u \xi, \dots, J_u^{\ell - 1} \xi \} \subset \g_1, \quad \forall \, \ell \in \N
\end{align}
with $U^0(p)$ defined to be $\{0\}$, and 
\begin{align}\label{spg2}
U_\ell(p) := \{v \in \g_2 : J_v (U^\ell(p)) = \{0\}\} \subset \g_2, \quad \forall \, \ell \in \N \cup \{0\}.
\end{align}
Note that $\{U^{\ell}(p)\}_{\ell \in \N \cup \{0\}}$ is increasing and $\{U_\ell(p)\}_{\ell \in \N \cup \{0\}}$ is decreasing. For $\ell \in \N \cup \{0\}$, let $W_\ell(p)$ be the orthogonal space of $U_{\ell+1}(p)$ in $U_\ell(p)$ and
\begin{align}\label{Wi}
	W_\infty(p) :=  \left\{v \in \g_2 : 	J_v(J_u^{\ell}\xi) = 0, \forall \, \ell \in \N	 \cup \{0\} \right\}.
\end{align}
By definition we have  the orthogonal splitting 
\[
\g_2 = \bigoplus_{\ell=0}^\infty W_\ell(p) \oplus W_\infty(p).
\]
Since $\g_2$ is finite dimensional, only finitely many $W_\ell(p)$'s are non-trivial. Finally define 
\begin{align}\label{eqnp}
N(p) := \begin{cases}
2Q - n + 2 \sum_{\ell = 0}^\infty \ell \, \dim (W_\ell (p)), \quad &\mbox{if $W_\infty(p) = \{0\}$}\\
\infty,  \quad &\mbox{if $W_\infty(p) \ne \{0\}$}
\end{cases}.
\end{align}
Recall that $Q= \dim \g_1 + 2 \dim \g_2 = m + 2 \dim \g_2$ and $n = \dim \g_1 + \dim \g_2 = m +  \dim \g_2$. The following theorem is the main result of \cite{NZ24} in the setting of step-two Carnot groups.

\begin{remark}
The number $N(p)$ defined here is exactly the number $\mathcal{N}_p$ in \cite[Definition 5.44]{ABR18} after the identification of $\g = T_e \G$ with $\g^* = T_e^* \G$ induced by the given inner product. For more details, we refer to \cite[Remark 1.4 and Remark 6.9]{NZ24}.
\end{remark}

\begin{theorem}[Theorem B and Theorem 6.7 of \cite{NZ24}]\label{tnz24}
In the setting of step-two Carnot groups, we have
\begin{align}\label{eqnce}
N_{\CE} \ge N_0 := \sup\{N(p) : p \in \g \mbox{ with } N(p) < \infty\}.
\end{align}
\end{theorem}

\begin{definition}[M\'etivier group]\label{defM}
We call a step-two Carnot group $\G$ a M\'etivier group if $J_v$ is invertible when $v \ne 0$.
\end{definition}

\begin{remark}\label{idealM}
This kind of groups is first introduced by M\'etivier in \cite{M80}.  In the setting of step-two Carnot groups, the condition of M\'etivier group is equivalent to requiring that there is no non-trivial abnormal length minimizing geodesics (cf. \cite[Proposition 3.6]{MM16} or \cite[Equation (3.2)]{LDMOPV16}). As a result, the notion of M\'etivier groups is equivalent to that of ideal Carnot groups. 
\end{remark}

By definition, it is easy to prove the following lemma.

\begin{lemma}\label{lemM}
If $\G$ is a M\'etivier group, then $N_0 = 2Q - n = \dim \g_1 + 3 \dim \g_2$.
\end{lemma}

\begin{proof}
From definition we know that if $U^{\ell}(p) \ne \{0\}$, then $U_{\ell}(p) = \{0\}$. For fixed $\g \ni p = \xi + u$ with $\xi \in \g_1$ and $u \in \g_2$ we know that $U^{\ell}(p) = \{0\}$ if and only if $\ell = 0$, or $\xi = 0$. As a result, we have 
\[
W_\infty(p) = \begin{cases}
\g_2, \quad &\mbox{if $\xi = 0$} \\
\{0\}, \quad &\mbox{if $\xi \ne 0$}
\end{cases},
\]
and if $\xi \ne 0$,
\[
W_\ell(p) = \begin{cases}
\g_2, \quad &\mbox{if $\ell = 0$} \\
\{0\}, \quad &\mbox{if $\ell \in \N$}
\end{cases}.
\]

From the \eqref{eqnp} above, we obtain
\[
N(p) = \begin{cases}
\infty, \quad &\mbox{if $\xi = 0$} \\
2Q - n, \quad &\mbox{if $\xi \ne 0$}
\end{cases},
\]
and thus $N_0 = 2Q - n$ by \eqref{eqnce}.
\end{proof}

\section{Proof of Theorem \ref{c1}}\label{sec4}

\begin{proof}

On a vector space $\g = \g_1 \oplus \g_2$ with an inner product $\langle \cdot, \cdot \rangle$ w.r.t. which $\g_1$ is orthogonal to $\g_2$, we choose an orthonormal basis of $\g_1$ and $\g_2$ respectively by
\[
\{\X_0, \X_1, \X_2, \X_3\}\quad \mbox{and} \quad \{\Y_1, \Y_2, \Y_3\}.
\]
Now we consider the following sequence of Lie algebra structure on $\g$ defined by the following relations:
\begin{align*}
[\X_0, \X_1]_k = \Y_1, \quad  &[\X_0, \X_2]_k = \Y_2, \quad [\X_0, \X_3]_k = \Y_3, \\
[\X_1, \X_2]_k = \frac{1}{k}\Y_3, \quad &[\X_1, \X_3]_k = -\frac{1}{k}\Y_2, \quad  [\X_2, \X_3]_k = \frac{1}{k}\Y_1,
\end{align*}
where $k \in \ol{\N}$ and $\frac{1}{\infty}$ is interpreted as $0$. In the following for fixed $k \in \ol{\N}$ we use $N_{0,\G_k}$, $N_{\CE, \G_k}$ and $J_{u,\G_k}$ to denote the lower bound given in \eqref{eqnce}, the curvature exponent and the operator defined by \eqref{defJ} on $\G_k = (\g,\ast_k)$ respectively. Note that when $k = \infty$, it is exactly the step-two Carnot group induced by the star graph $K_{1,3}$ in \cite[\S~7.3]{NZ24} on which $N_{0,\G_\infty} = 17$.  It follows from the construction and Remark \ref{sc} that $\{(\g,[\cdot,\cdot]_k)\}_{k \in \N}$ converges to $(\g,[\cdot,\cdot]_\infty)$ as $k \to \infty$. Now let $\g_2 \ni u = \sum_{\ell = 1}^3 u_\ell \Y_\ell$, under the basis $\{\X_0, \X_1, \X_2, \X_3\}$, we have 
\[
J_{u,\G_k} = \begin{pmatrix}
0 & - u_1 & - u_2 & - u_3 \\
u_1 & 0 & -\frac{1}{k} u_3 & \frac{1}{k} u_2 \\
u_2 & \frac{1}{k} u_3 & 0 & -\frac{1}{k} u_1 \\
u_3 & -\frac{1}{k} u_2 & \frac{1}{k} u_1 & 0
\end{pmatrix}.
\]
By taking square $J_{u,\G_k}^2$ we know that $J_{u,\G_k}$ is invertible if $u \ne 0$ and $k \in \N$, which means  $\G_k (k \in \N)$ is M\'etivier, or  equivalently ideal by Remark \ref{idealM}. As a result, we know that $N_{0,\G_k} = 13$ for $k \in \N$ by Lemma \ref{lemM}. Assume on the contrary that $N_{\CE} = N_0$ holds for all ideal Carnot groups. Then we know $N_{\CE, \G_k} = N_{0,\G_k} = 13$ for $k \in \N$. It follows from Proposition \ref{t1} that
\[
17 = N_{0,\G_\infty}  \le N_{\CE, \G_\infty} \le\liminf_{k \to \infty} N_{\CE,\G_k} = \liminf_{k \to \infty} N_{0,\G_k} = 13,
\]
which leads to a contradiction. We complete the proof. In particular, we have showed that $N_{0}$ is not lower semicontinuous w.r.t. the Lie algebra convergence.

\end{proof}

\section*{Acknowledgements}

\medskip

The author would like to thank the anonymous referee for many valuable suggestions and remarks, especially the reference to the paper \cite{ALN23}, which reduces the length of the proof significantly. The author would also like to thank Kenshiro Tashiro for fruitful discussions.

\end{document}